\documentclass{amsart}
\usepackage{color,graphicx,wrapfig}
\usepackage{eucal}

\title{Harmonic balls and the two-phase Schwarz function}
\author{ Henrik Shahgholian}
\address{Department of Mathematics, Royal Institute of Technology,
100 44 Stockholm Sweden}
\email{henriksh@math.kth.se}

\author{Tomas Sj\"o{}din}
\address{Department of Mathematics, Link\"o{}ping University, 581 83 Link\"o{}ping, Sweden}
\email{tosjo@mai.liu.se}
\thanks{The first author is partially supported by the Swedish Research Council.\newline
 { The authors wish to thank professor Bj\"orn Gustafsson for valuable comments and fruitful discussions. In particular example \ref{Schottky} was contributed by him.}}
  

\newtheorem{theo}{Theorem}[section]
\newtheorem{prop}[theo]{Proposition}
\newtheorem{lem}[theo]{Lemma}
\newtheorem{defi} {Definition}[section]
\newtheorem{cor}[theo] {Corollary}
\newtheorem{rem}[theo]{Remark}

\newtheorem{conj} {Conjecture}[section]
\newtheorem{quest}{Question}[section]
\newtheorem{example}{Example}[section]
\newcommand{\R}{\mathbb{R}}
\newcommand{\C}{\mathbb{C}}

\begin{document}

\begin{abstract}
Here we shall introduce the concept of harmonic balls in sub-domains of $\R^n$, through a  mean value property for
a sub-class of harmonic functions on such domains. In the complex plane, and for analytic functions, a similar concept fails to exist
due to the fact that analytic functions can not have prescribed data on the boundary. Nevertheless, a two-phase version of the problem
does exists, and gives rise to the generalization of the well-known Schwarz function to the case of a
two-phase Schwarz function.
Our primary goal is to derive simple properties for these problems, and tease the appetites of experts working on Schwarz function and related topics.
Hopefully these two concepts will  provoke further study of the topic.
\end{abstract}

\maketitle

%%%%%%%%%%%%%%%%%%%%%%%%%%%%%%%%%%%%%%%%%%%%%%%%
%%%%%%%%%%%%%%%%%%%%%%%%%%%%%%%%%%%%%%%%%%%%%%%%%
\section{Introduction and basic Notation}
We will be working in $\R^n$ ($n\geq2)$ and for a (signed) Radon measure $\mu$ with compact support in $\R^n$ we let $U \mu$ denote the Newtonian/logarithmic potential normalized so that $-\Delta U \mu = \mu$ in the sense of distributions. We will by $\delta_x$ denote a point mass at $x$ and by $\lambda$ we denote the Lebesgue measure. If furthermore $K \subset \R^n$ is Greenian (i.e. an open set which has a Green's function), then we let $G_K(\cdot,\cdot)$ denote its Green function, and we denote Green potentials by $G_K \mu$. We will regard this function as defined on all of $R^n$ and identically zero on $K^c$. We will also denote the sweeping (balayage) of a finite measure $\mu$ in $K$ onto $\partial K$ by $\mu^{K^c}$. It is defined in such a way that if $h$ is harmonic in $K$ and continuous on $\overline{K}$ then $\int h d \mu = \int h d (\mu^{K^c}),$ and it is furthermore given by $\mu^{K^c}= (\Delta G_K \mu) |_{K^c}$.
We will denote the open ball with center $x^0$ and radius $r$ by $B_r(x^0)$.
 
The paper will in a sense consist of two parts which are closely related. Both concern questions which to a large extent are motivated by the theory of quadrature domains and related topics. In the first part, which consists of section 2 and 3, we will study what we call {\bf harmonic balls} in a subdomain $K \subset \R^n$. These are subsets $D=D(x^0,\alpha)$ of $K$ which satisfy the mean value property
$$\alpha h(x^0) = \int_{D} h d \lambda,$$
for all functions $h$ which are harmonic in $D$ and (roughly speaking) vanish on $\partial D \cap \partial K$.
Here one could of-course in analogy with the theory of quadrature domains study more general measures than point masses, but most of the questions we are interested in here would have negative answers in more general settings (however the basic existence results, for instance, can be proved in the same manner with small changes under more general assumptions).

The second part (section 4), which will deal with what we call {\bf two-phase Schwarz functions}, can be motivated in two ways.
The first motivation is the failure to generalize the concept of harmonic balls to analytic balls, which is due to that an analytic function which vanishes on some non-trivial part of the boundary of for instance a simply connected domain has to be identically zero. But it turns out that one can still in some sense generalize the idea to two phases, where one balances the values of the function from both sides of the boundary.

The other motivation comes directly from the recent theory of two-phase quadrature domains (see \cite{Em-Pr-Sh 2010,GS3}).
Let us first recall that a (one-phase) quadrature domain for harmonic (analytic) functions in the plane is a bounded open set $D \subset \R^2$ such that for some distribution $\mu$ with compact support in $D$ we have 
$$\langle \mu, h \rangle = \int_D h d \lambda$$
for all integrable harmonic (analytic) functions $h$ in $D$. This is equivalent to that the function $u= U \mu - U (\lambda |_D)$ satisfies $u = |\nabla u|=0$ (respectively just $|\nabla u|=0$) in $D^c$. If we define the function $S =\overline{z} - 4 \partial u$ and $D$ is a quadrature domain, then $S$ satisfies $\overline{\partial} S = \mu$ in $D$ and $S(z) = \overline{z}$ on $\partial D$. In particular $S$ is analytic in $D \setminus \textrm{supp} (\mu)$. This is the definition of a one-sided (one-phase) Schwarz function with respect to $\partial D$.  Here it is enough that $D$ is a quadrature domain for analytic functions with respect to $\mu$. Furthermore this can be reversed. If for some domain $D$ there exists a one-sided Schwarz function $S$, i.e. a function $S$ which equals $\overline{z}$ on $\partial D$ and is analytic in $D \setminus C$ for some compact subset $C \subset D$, then $\overline{\partial} S = \mu$ has compact support in $D$ and $u= U \mu - U(\lambda |_D)$ satisfies $S =\overline{z} - 4 \partial u$, so in particular $|\nabla u|=0$ on $\partial D$. Hence $D$ is a quadrature domain for analytic functions with respect to $\mu$. 

The two-phase Schwarz function will be similarly related to two-phase quadrature domains. We recall that a two-phase quadrature domain for harmonic functions with respect to the pair $(\mu_+,\mu_-)$ of measures with disjoint compact supports and positive numbers $(\beta_+,\beta_-)$ is a pair of disjoint bounded open sets $(D_+, D_-)$ such that $\textrm{supp} (\mu_{\pm}) \subset D_{\pm}$ and the function 
\begin{equation}\label{little u}
u = (U \mu_+ - \beta_+ U (\lambda |_{D_+}))-(U \mu_- - \beta_- U (\lambda |_{D_-}))
\end{equation}
vanishes in $D^c$, where $D=D_+ \cup D_-$. (Note that we do not assume that the gradient vanishes which it will typically not do on $\partial D_+ \cap \partial D_-$). This is roughly equivalent to a two-phase quadrature identity for all harmonic functions which are continuous up to the closure of $D_+ \cup D_-$.
The function
$$S = \left\{ \begin{array}{lr} \beta_+ \overline{z} - 4 \partial u & \textrm{ in } D_+\\
								 -\beta_- \overline{z} - 4 \partial u & \textrm{ in } D_-\end{array}\right.,$$
satisfies 
$$\overline{\partial}S = \left\{ \begin{array}{rl} \beta_+ - \Delta u = \mu_+ & \textrm{ in } D_+\\
														-\beta_- - \Delta u = - \mu_- & \textrm{ in } D_-\end{array}\right.,$$
and (roughly speaking)
\begin{eqnarray*} &&S(z) = \beta_+\overline{z}, \quad z \in \partial D_+ \setminus \partial D_-\\
							&&S(z) =	-\beta_- \overline{z}, \quad z \in \partial D_- \setminus \partial D_+\\
							&&\lim_{w \rightarrow z, w \in D_+} S(w) - \lim_{w \rightarrow z, w \in D_-} S(w) =	(\beta_+ + \beta_-)\overline{z}, \quad z \in \partial D_+ \cap \partial D_-.
\end{eqnarray*}

Such a function will be an example of a two-phase Schwarz function for $(D_+,D_-)$.
We should note here that the gradient $\nabla u$ does not vanish on $\partial D_+ \cap \partial D_-$ but we have 
that the limit of the gradient from $D_+$ equals - the limit of that from $D_-$.

We end the paper with a short section listing possible areas for future research within this field.
\subsection{Notation}
\begin{itemize}
\item[$\cdot$] $U \mu$ : Newtonian (logarithmic if $n=2$) potential of $\mu$,
\item[$\cdot$] $G_K \mu$ : Green potential of $\mu$,
\item[$\cdot$] $\Delta$ : Laplacian,
\item[$\cdot$] $B_r (x)$ : open ball with center $r$ and radius $x$,
\item[$\cdot$] $\lambda$ : Lebesgue measure,
\item[$\cdot$] $\overline{D}$ : closure of $D \subset \R^n$,
\item[$\cdot$] $D^o$ : interior of $D \subset \R^n$.
\end{itemize} 
In $\R^2=\C$ we will also use the notation 
$$\partial = \frac{\partial}{\partial z},\quad \overline{\partial}=\frac{\partial}{\partial \overline{z}}.$$ 
\section {Harmonic balls}

Let $K\subset \R^n$ ($n\geq 2$) be a Greenian domain (open connected set). For a sub-domain $D \subset K$ we define 
$$\widetilde{H}_K(D) = \{ G_K \mu: \mu \textrm{ is a signed Radon measure with compact support in } K \setminus D\},$$
$$\widetilde{S}_K(D) = \{ G_K \mu: \mu \textrm{ is a signed Radon measure with compact support in } K, \mu |_D \leq 0  \}.$$

\begin{defi}[Harmonic/Subharmonic balls] Let $x^0 \in K$ and $\alpha >0$.

A subset $D(x^0,\alpha) \subset K$ is called a harmonic ball relative to $K$ if 
\begin{equation}\label{HB}
\int_{D(x^0,\alpha)} h (x) \ d\lambda = \alpha h(x^0), \qquad \forall h\in \widetilde{H}_K(D(x^0,\alpha)).
\end{equation}

A subset $D(x^0,\alpha) \subset K$ is called a subharmonic ball relative to $K$ if 
\begin{equation}\label{SB}
\int_{D(x^0,\alpha)} s (x) \ d\lambda \leq \alpha s(x^0), \qquad \forall h\in \widetilde{S}_K(D(x^0,\alpha)).
\end{equation}
We will refer to $x^0$ as the center of the ball, and $\alpha$ as the size of the ball.
\end{defi}

If $\overline{D(x^0,\alpha)} \subset K $ is a harmonic ball then it coincides with the standard ball with center $x^0$ and Lebesgue measure $\alpha$.
Indeed, this follows from the well-known mean value property  for harmonic functions over the balls (see \cite{gaier}), 
and the fact that balls are the only domains with this property (see \cite{ep}). 
Also, if $K$ were not Greenian then the only domains reasonably corresponding to harmonic balls would be of the form $B_r(x) \cap K$, and hence it is of no loss of generality to assume that $K$ is Greenian.

Subharmonic balls has also been introduced and studied by M. Sakai (see \cite{SA}), where they are called restricted quadrature domains.
\begin{theo}
Let $x^0 \in K$, $\alpha>0$ and $D \subset K$ be open, and define 
$$u_K = G_K(x^0,\cdot)- G_K (\lambda |_D).$$
Then the following holds:
\begin{itemize}
\item[(a)] D is a harmonic ball with center $x^0$ and size $\alpha$ if and only if $u_K =0$ in $K \setminus D.$
\item[(b)] D is a subharmonic ball with center $x^0$ and size $\alpha$ if and only if $u_K \geq 0$ in $K$ and $u_K =0$ in $K \setminus D.$  
\end{itemize}
\end{theo}
\begin{proof}
The only if statements are clear since $G_K(\cdot,x) \in \widetilde{H}_K(D)$ if $x \in K \setminus D$ and $-G_K(\cdot,x) \in \widetilde{S}_K(D)$ for all $x \in K$.
In the other direction we note that if $\mu$ has compact support in $K \setminus D$ and $u_K =0$ in $K \setminus D$ then we have by Fubini's theorem
$$\alpha G_K \mu (x^0) - \int_D G_K \mu d\lambda = \int (\alpha G_K \delta_{x^0} - G_K (\lambda |_D)) d \mu =0,$$
because the integrand on the right hand side is identically zero on the support of $\mu$.

If we instead only assume that $\mu \leq 0$ in $D$, and that $u_K \geq 0$ with equality in $K \setminus D$, then the same type of argument gives
$$\alpha G_K \mu (x^0) - \int_D G_K \mu d\lambda = \int (\alpha G_K \delta_{x^0} - G_K (\lambda |_D)) d \mu \leq  0,$$
because the integrand is nonnegative in $K$, zero outside of $D$ and $\mu \leq 0$ in $D$ by assumption.
\end{proof} 
\begin{rem}
We note that in case $K=D$, then the class $\widetilde{H}_K(D)$ only contains the zero function.
Hence $K$ itself will always be a harmonic ball with our definition. We will call this the trivial harmonic ball. The problem is that if $K$ is small compared to $\alpha$, then $K$ will also be the only candidate for a harmonic ball, and therefore we do not wish to exclude it either. But care has to be taken when formulating uniqueness results due to this.

It is also a bit unclear to what extent the definition, even if overlooking this trivial case, is enough in general to guarantee some sort of uniqueness in general. It does so in many cases, for instance of $K$ is a half-space as we will see later.
Due to this it might be natural to introduce some extra condition on harmonic balls. One such condition that seems natural is the following:
\end{rem}
\begin{defi}
A harmonic ball $D(x^0,\alpha)$ is said to be positive if the sweeping $(\alpha \delta_{x^0} - \lambda |_{D(x^0,\alpha)})^{K^c}$ of $\alpha \delta_{x^0} - \lambda |_{D(x^0,\alpha)}$ onto $\partial K$ is positive.
\end{defi}
It is implicit in this definition that the sweeping is well defined (i.e. $\lambda(D(x^0,\alpha)) < \infty$). It follows immediately that the assumption implies that $\lambda (D(x^0,\alpha)) \leq \alpha$ if $D(x^0,\alpha)$ is a positive harmonic ball. In particular if $\lambda(K) > \alpha$ then $K$ is not a positive harmonic ball with size $\alpha$ for any point in $K$.

It is furthermore clear that a subharmonic ball is always a positive harmonic ball in this sense.
\begin{quest}
Are there any non-trivial examples of harmonic balls that are not positive?
\end{quest}
\section{Existence and Uniqueness}
In this section we will first recall some basic facts about partial balayage, and use this to prove existence of subharmonic balls.
After this we study the uniqueness of a harmonic ball with given center $x^0$ and  size $\alpha$. 
The uniqueness of harmonic balls for $K=\R^n$ is well known (see e.g. \cite{Shahgholian 92a}).

\subsection{Partial balayage}
Here we recall some basic facts about the notion of (one-phase) partial
balayage, which was originally developed by Gustafsson and Sakai \cite{GusSa}%
. A recent exposition of it may be found in \cite{GS2}. For an open set $%
K\subset {\mathbb{R}}^{N}$ and a positive measure $\mu $ with compact
support in $K$ we define 
\begin{equation*}
V_{K}{\mu (x)}=\sup \left\{ v(x):v\text{ is subharmonic on }K\text{ and }%
v\leq U\mu +\frac{\left\vert \cdot \right\vert ^{2}}{2N}\text{ on }\mathbb{R}%
^{N}\right\} -\frac{\left\vert x\right\vert ^{2}}{2N}
\end{equation*}%
and then put $B_{K}\mu =-\Delta V_{K}\mu $. It turns out that there is a
measure $\nu $ such that 
\begin{equation}
B_{K}\mu =\lambda |_{\omega (K,\mu )}+\mu |_{\omega (K,\mu )^{c}}+\nu
=\lambda |_{\Omega (K,\mu )}+\mu |_{\Omega (K,\mu )^{c}}+\nu ,  \label{pb}
\end{equation}%
where 
\begin{equation*}
\omega (K,\mu )=\{V_{K}\mu <U\mu \}
\end{equation*}%
and 
\begin{equation*}
\Omega (K,\mu )=\bigcup \{U:U\subset K\text{ open and }B_{K}\mu =\lambda 
\text{ in }U\},
\end{equation*}%
and these are bounded open subsets of $K$. (Clearly $V_{K}\mu =U\mu $ on $%
K^{c}$.) Further, 
\begin{equation}
B_{K}\mu \leq \lambda \text{ \ on }K\text{\ \ \ and \ \ }\nu \geq 0,
\label{pb2}
\end{equation}%
and $\nu $ is supported by $\partial K\cap \partial \omega (K,\mu )$. We
note that $\omega (K,\mu )\subset \Omega (K,\mu )$ and that this inclusion
may be strict, even when $\mu $ has compact support contained in $\Omega
(K,\mu )$. Clearly these sets increase as $K$ increases and as $\mu $
increases. It will be convenient to define 
\begin{equation*}
W_{K}\mu =U\mu -V_{K}\mu ,
\end{equation*}%
whence $W_{K}\mu $\ is lower semicontinuous, $-\Delta W_{K}\mu \geq \mu
-\lambda $ on $K$ and $W_{K}\mu \geq 0$ on $\mathbb{R}^{N}$. Finally, if $K={%
{\mathbb{R}}^{N}}$, we will abbreviate the above notation to $V\mu $, $B\mu $%
, $\omega (\mu )$, $\Omega (\mu )$, and $W\mu $, respectively. In this case, 
$\nu =0$.

\subsection{Existence and uniqueness}
\begin{prop} \label{existsh}
For every $x^0 \in K$ and every $\alpha >0$ there is up to a Lebesgue null set a unique subharmonic ball $D(x^0,\alpha)$.
\end{prop}
\begin{proof}
We note that if we let $D(x^0,\alpha) = \omega(K,\alpha \delta_{x^0})$ then it follows by construction that it is a subharmonic ball as stated.
We immediately get that $\lambda(\Omega(K,\alpha \delta_{x^0}) \setminus \omega(K,\alpha \delta_{x^0}))=0.$ 
The uniqueness will follow from proposition \ref{shdom} below.
\end{proof}
\begin{quest}
Are there any examples of $K,x^0,\alpha$ such that $\omega(K,\alpha \delta_{x^0}) \ne \Omega(K,\alpha \delta_{x^0})$?
\end{quest}
\begin{prop} \label{shdom}
Let $D$ be a harmonic ball with center $x^0$ and size $\alpha$ and $\Omega$ a subharmonic ball with center $x^0$ and size $\alpha$ in $K$. Then $G_K \lambda |_{\Omega} \leq G_K \lambda |_{D}$ in $K$ and $(\lambda |_D)^{K^c} \geq (\lambda |_{\Omega})^{K^c}$. Furthermore $\partial D \cap K \subset \overline{\Omega}.$
\end{prop}
\begin{proof}
Since $G_K \lambda |_D = G_K \delta_{x^0} \geq G_K \lambda |_{\Omega}$ in $K \setminus D$ we see that $u = G_K \lambda |_D -G_K \lambda |_{\Omega} \geq 0$ in $K \setminus D$. And in $D$ we have $-\Delta u = \chi_D - \chi_{\Omega} \geq 0$, so $u$ is superharmonic in $D$. Hence $u \geq 0$ by the maximum principle. By Kato's inequality we have
$$(\lambda |_D)^{K^c} = (\Delta G_K (\lambda |_D)) |_{K^c} \geq (\Delta G_K (\lambda |_{\Omega})) |_{K^c} = (\lambda |_{\Omega})^{K^c}.$$

For the second part, assume that $x \in (\partial D \setminus \overline{\Omega}) \cap K$. Then $u$ is superharmonic and not identically zero close to $x$. Furthermore it is nonnegative and assumes its minimum value $0$ in $x$, which contradicts the minimum principle.
\end{proof}
\begin{rem}
Clearly the above proposition implies that subharmonic balls are unique up to a Lebesgue null set, since they have to produce the same Green potential. Indeed, it follows more or less directly from our definitions that we have 
$$\omega(K,\alpha \delta_{x^0}) \subset 
D(x^0,\alpha) \subset \Omega(K,\alpha \delta_{x^0})$$
for all subharmonic balls $D(x^0,\alpha)$.

It is also well known that $\lambda (\partial \Omega(K,\alpha 
\delta_{x^0}) \cap K)= \emptyset$, and hence it follows that $\Omega(K,\alpha \delta_{x^0})= (\overline{\omega(K,\alpha 
\delta_{x^0})})^o \cap K$. 
\end{rem}
\begin{cor}
Suppose that $\Omega$ is a subharmonic ball with center $x^0$ and size $\alpha$. If $K \setminus \Omega$
is connected, then any harmonic ball $D$ with center $x^0$ and size $\alpha$ such that $D \cup \overline{\Omega} \ne K$ satisfies $D \subset \Omega(K,\alpha \delta_{x^0})$. In particular, if $K$ is a half-space then the only harmonic ball with center $x^0$ and size $\alpha$, apart from the trivial one, is the subharmonic ball $\omega(K,\alpha \delta_{x^0})$
\end{cor}
\begin{proof}
We note that it is enough to prove that $D \subset \overline{\Omega}$, because it is well known that  $\partial \Omega \cap K$ has zero Lebesgue measure, and hence $\Omega(K,\alpha \delta_{x^0})$ equals the interior of $\overline{\Omega}$. 
Since $K \setminus \Omega$ is connected we have that if $x \in D \setminus \overline{\Omega}$, then either $(\partial D 
\setminus \overline{\Omega}) \cap K$ is nonempty which contradicts the fact that $\partial D \cap K \subset \overline{\Omega}$
, or $D$ contains $K \setminus \overline{\Omega}$ which is not the case by assumption.

For the last part we simply note that it is easy to see that if $K$ is a half-space, then $\omega(K,\alpha \delta_{x^0}) =\Omega(K,\alpha \delta_{x^0})$ and also that $K \setminus \omega(K,\alpha \delta_{x^0})$ is connected.
\end{proof}
\begin{quest}
Are there any examples of non-trivial harmonic balls which are not subharmonic balls?
\end{quest}
We will now prove that in case $K$ is starshaped with respect to $x^0$ then so are the subharmonic balls centered at $x^0$.
\begin{theo}
If $K$ is starshaped with respect to $x^0$, then so is $\omega(K,\alpha \delta_{x^0})$. In particular $\Omega(K,\alpha \delta_{x^0}) = \omega(K,\alpha \delta_{x^0})$.
\end{theo}
\begin{proof} 
Without loss of generality we may assume that $x^0 =0$. Furthermore if we exhaust $K$ by domains $K_n$, then it is easy to see that $\omega(K_n,\alpha \delta_{x^0})$ increases to $\omega=\omega(K,\alpha \delta_{x^0})$. Hence we may without loss of generality assume that $K$ has a smooth boundary.

Now let 
$$u = \alpha G_K(x^0,\cdot) - G_K (\lambda |_{\omega}) \textrm{ in } K,$$
and
$$w(x)= x \cdot \nabla u(x) \textrm{ in } K \setminus \{0\}.$$
Since $\partial K$ is smooth the function $u$ suitably extended to $K^c$ belongs to $C^1(\R^n),$ which will be used below.

The function $w$ satisfies
$$\Delta w (x) = x \cdot \nabla (\Delta u)(x) + 2 \Delta u (x) = 2 \textrm{ in } \omega \setminus \{0\}.$$
Hence it is subharmonic in $\omega \setminus \{0\},$ and furthermore it is clear that close to $0$ $w$ is negative (since $u$ goes towards infinity as we approach $0$).
On the other hand, if $x \in \partial \omega$, then $x \cdot \nabla u(x)$ is non-positive. This follows since if $x \in \partial \omega \cap K$, then $\nabla u(x)=0$, and hence $w(x)=0$, and if $x \in \partial K \cap \partial \omega$
then we have
$$w(x) = x \cdot \nabla u(x) = |x| \lim_{h \rightarrow 0^+} \frac{u(x-hx)-u(x)}{-h} \leq 0,$$
(above we used that $K$ is starshaped).
From the strong maximum principle it follows that $w$ is strictly negative in $\omega \setminus \{0\},$ and from this it follows immediately that $\omega$ must be starshaped with respect to $0$, because if $x \in \omega$ then the set
$\{t: 0 \leq t \leq 1, tx \in \omega\}$ contains all points close to $0$ and $1$, and must be connected. 
\end{proof}
We end this section with some results for positive harmonic balls.
\begin{lem}  \label{pmeas} Suppose $D \subset K$ is open and that there is an open set $T \subset \R^n$ such that $L = T \cap \partial D = T \cap \partial K$ and  $L$ is non-polar. Then for every $t >0$ we have that 
$$((1+t) \lambda |_D - \lambda |_{\omega(K,(1+t) \lambda |_D)})^{K^c}(L) >0.$$
\end{lem}
\begin{proof}
let $\mu = (t \lambda |_D)^{D^c}.$  Then we have
$$G_K(\lambda |_D +t \mu |_K) = G_K ((1+t) \lambda |_D) \textrm{ in } K \setminus D.$$
This follows because by definition we have
$$w=U (\lambda |_D + t \mu ) - U((1+t) \lambda _D) =0 \textrm{ in } D^c.$$
Since this function satisfies $w=0$ on $\partial K$ and $-\Delta w = \lambda |_D + t \mu -(1+t) \lambda |_D,$
we see that it equals $G_K(\lambda _D +t \mu |_K) - G_K ((1+t) \lambda_D) \textrm{ in } K$.

Now if $v$ satisfies $v \leq G_K((1+t)\lambda |_D)$ and $-\Delta v \leq 1$ in $K$, then it follows immediately from the maximum principle that $v \leq G_K( \lambda |_D + t \mu |_K)$ in $K$. Hence we have
$$ B_K (\lambda |_D + t \mu |_K) + t \mu |_{K^c} = B_K ((1+t) \lambda |_D).$$
In particular $B_K ((1+t) \lambda |_D) |_{K^c} \geq t \mu |_{K^c}$.
Since the harmonic measure of $L$ with respect to $D$ is easily seen to be positive under the stated hypothesis it also follows that $\mu (L) >0$, and hence the lemma is proved.
\end{proof}
\begin{theo}
If every component of $\Omega(K,\alpha \delta_{x^0})^c$ contains some non-polar component of $K^c$ then $D(x^0,\alpha) \subset \Omega(K,\alpha \delta_{x^0})$ for every positive harmonic ball $D(x^0,\alpha)$.  
\end{theo}
\begin{proof}
Let $D = D(x^0,\alpha)$ and $\Omega_t = \Omega(K,(1+t)\alpha \delta_{x^0}).$
We know that $\partial D \subset \overline{\Omega_0}$, and hence either $D \subset \Omega_0$, or there must be some component
$S$ of $K \setminus \overline{\Omega_0}$ which is contained in $D$. We assume the latter to derive a contradiction.

Now we note that for every $t >0$ the set
$$\omega_t = \omega(K,(1+t) \lambda |_D)$$
is also a positive harmonic ball with respect to $x^0$ and size $(1+t)\alpha$. To see this we simply note that 
$$(1+t) (\alpha G_K (x^0,\cdot) - G_K(\lambda |_D)) \leq (1+t) \alpha G_K(x^0,\cdot) - G_K (\lambda |_{\omega_t}).$$
By Kato's inequality it follows that the Laplacian of the right hand side dominates that of the left hand side on $\partial K$, because the function $w=(1+t) \alpha G_K(x^0,\cdot) - G_K (\lambda |_{\omega_t}) - (1+t) (\alpha G_K (x^0,\cdot) - G_K(\lambda |_D))$ is positive in $K$ and zero on $K^c$, so $(\Delta w) |_{K^c} \geq 0$. 

But it is also clear that for small $t$ the set $\Omega_t$ does not contain $S$. Indeed it is easy to see that for small $t$ there is a point $ y \in \partial K$ and $\varepsilon >0$ such that $B_{\varepsilon}(y) \cap K \subset S$, $B_{\varepsilon}(y) \cap \overline{\Omega_t} = \emptyset$ and $L = \partial K \cap B_{\varepsilon}(y)$ is non-polar.

We also know by proposition \ref{shdom} that 
$$(1+t) (\alpha G_K (x^0,\cdot) - G_K(\lambda |_{\omega_t})) \leq (1+t) \alpha G_K(x^0,\cdot) - G_K (\lambda |_{\Omega_t}).$$
But as in the previous argument it follows from Kato's inequality that the Laplacian of the left hand side restricted to $L$ must be zero. This contradicts lemma \ref{pmeas}.
\end{proof}
An immediate consequence of the above results is the following:
\begin{cor}
If $K$ is starshaped with respect to $x^0$, then there is only one positive harmonic ball $D(x^0,\alpha)$ with center $x^0$and size $\alpha$.
\end{cor}
\begin{quest}
Is there any case where there is a non-polar component of $\Omega(K,\alpha \delta_{x^0})^c$ which does not contain a non-polar component of $K^c$?
For instance if $K$ contains no holes, can we draw the same conclusion about $\Omega(K,\alpha \delta_{x^0})$?
\end{quest}
\begin{conj}
Every positive harmonic ball is a subharmonic ball.
\end{conj}
At the very least this seems reasonable if $K$ satisfies weaker conditions than being starshaped as above.
\section {The Two-Phase Schwarz function}
In this section we will be working in $\R^2=\C$ and it is natural to use complex notation $z=x+iy$. To begin with let $K \subset \C$ be Greenian and let $D=D(z^0,\alpha)$ be a harmonic ball in $K$. We also assume that $D$ is sufficiently regular so that the function $u= \alpha G_K(z^0,\cdot) - G_K(\lambda |_D)$ satisfies $u=|\nabla u|=0$ in $K \setminus D$. Now let $v = 4 \partial u$, so that 
$\overline{\partial} v = \Delta u = \lambda - \delta_{z^0}$ in $D$ and $v=0$ on $\partial D \cap K$. However we have little information about the behavior of $v$ on $\partial D \cap \partial K$.

If we try to have a mean-value property for $D$, with respect to analytic functions
then by setting $S=\bar z - v$, one using that $\bar \partial S= \delta_0$ and $S=\bar z$ on  $\partial D \cap K$, along with the complex-version of
Stoke's formula 
$$
\int_D f d\lambda = \int_{\partial D} f(z)\bar z dz = \int_{\partial D} f(z) S(z) dz  + I =
f(z^0) + I
$$
where $I = \int_{\partial D \cap \partial K} f(z)(\bar z - S(z)) dz$. Now to make sense of the above expression, we would like to have $I=0$ 
for a reasonable subclass of  analytic functions on $D$. If we  consider analytic functions in $D$ that vanish on
 $\partial D \cap \partial K$, then (as long as $\partial D \cap \partial K$ is not very small) this class only contains the zero function.   

 The question is whether this is the end of the story! Here we shall try to develop a two-phase version of the above problem, and circumvent the 
 above difficulty of defining holomorphic discs. Indeed, we shall show that one can still have a quadrature identity but in two phases. In this way
 we need not to assume zero boundary data for analytic functions (which reduces the class to the zero function). The balance of the boundary value
then comes from both sides of the boundary, and the boundary values are canceled. 
The concept needed for this is the definition of a two-phase Schwarz function. To define this we need two disjoint bounded open sets $D_+,D_-$ and two positive constants $\beta_+,\beta_-$. The definition will only be of interest if $\Gamma$ contains some curve/is not to small. 
\begin{defi}\label{tpsz}
Let $D_+,D_-$ be disjoint bounded open sets, and let $\beta_+,\beta_-$ be positive constants. Also let $D = D_+ \cup D_-$ and $\Gamma = \partial D_+ \cap \partial D_-$. Suppose that there are compact subsets $C_{\pm} \subset D_{\pm}$ and functions
$S_{\pm} \in C(\overline{D}_{\pm}) \cap A(D_{\pm} \setminus C_{\pm})$ such that
\begin{eqnarray*}
S_{\pm}(z) &=& \pm \beta_{\pm} \overline{z} \quad z \in \partial D_{\pm} \setminus \Gamma,\\
S_+(z) - S_-(z) &=& (\beta_+ + \beta_-)\overline{z} \quad z \in \Gamma,\end{eqnarray*}
then we say that the pair $(S_+,S_-)$ is a two-phase Schwarz function for $(D_+,D_-)$.
\end{defi}
\begin{rem}
All information of interest of the two-phase Schwarz function is really contained in the function 
$$S(z) = \left\{ \begin{array}{ll} S_+(z) & z \in \overline{D}_+ \setminus \Gamma\\
									  S_-(z) & z \in \overline{D}_- \setminus \Gamma, \end{array}\right.$$
and we note that
$$\lim_{w \rightarrow z, w \in D_+} S(w) - \lim_{w \rightarrow z, w \in D_-} S(w) =	(\beta_+ + \beta_-)\overline{z}, \quad z \in \partial D_+ \cap \partial D_-.$$
We will also refer to this function as the two-phase Schwarz function (the only difference is that we do not have any values on $\Gamma$, where the function $S$ typically would have discontinuities). 

Also note that if there are functions $S_{\pm} \in C(D_{\pm} \setminus C_{\pm}) \cap A (D_{\pm} \setminus C_{\pm})$ satisfying the boundary conditions in the above definition, then we may (enlarging $C_{\pm}$ slightly if necessary) extend $S_{\pm}$ to become continuous, and even smooth, in all of $D_{\pm}$. In particular we see that having a two-phase Schwarz function is a local property of the boundaries $\partial D_{\pm}$. 	

Furthermore note that by definition the distribution $\overline{\partial}S|_D$ has compact support in $D$. 
\end{rem}

Our interest in this function comes from its connection with two-phase quadrature domains for analytic functions which we now define.
\begin{defi}
The pair $(D_+,D_-)$ is said to be a ($(\beta_+,\beta_-)-$) two-phase quadrature domain for analytic functions with respect to the distribution $\mu$ if $\mu$ has compact support in $D= D_+ \cup D_-$ and 
$$\beta_+ \int_{D_+} f d \lambda - \beta_- \int_{D_-} f d \lambda = \langle \mu, f \rangle$$
for all $f \in A(D)\cap C(\overline{D}).$
\end{defi}
\begin{rem}
The choice of test function class $A(D)\cap C(\overline{D})$ compared to the class $AL^1$ used in the one-phase case is chosen due to that the gradient $\nabla u$ does not vanish on $\Gamma$ typically. The choice is in some respects not optimal (just like for the case of harmonic functions which is discussed in \cite{GS3}) in the sense that it does not give us a complete equivalence between the concept of having vanishing gradients as in the definition of the two-phase Schwarz function and having a two-phase quadrature domain for analytic functions. However this problem is not a major one in the sense that if the boundaries are smooth enough, then they are easily seen to be equivalent by approximation. 

It is furthermore clear that if $(D_+,D_-)$ is a two-phase quadrature domain for harmonic functions, then it is so also for analytic functions.
\end{rem}
The next theorem, which can be sharpened when it comes to the assumption of the regularity of $\partial D_{\pm}$, explains the connection between the two-phase Schwarz function and two-phase quadrature domains for analytic functions.  
\begin{theo}
If $(D_+,D_-)$ is a pair of disjoint bounded domains such that $\partial D_{\pm}$ are piecewise $C^1$, then it has a two-phase Schwarz function if and only if it is a two-phase quadrature domains for analytic functions with respect to some distribution.
\end{theo}
\begin{proof}
Let us first assume that $(D_+,D_-)$ has a two-phase Schwarz function $S$ and let $f$ be analytic in $D$ and continuous on $\overline{D}$.
By assumption $\overline{\partial} S = \mu$ has compact support in $D$. Now we get by the use of Stoke's theorem (where we now use that the boundaries are smooth enough):
\begin{eqnarray*}&&\int_{D_+} \beta_+ f d \lambda - \int_{D_-} \beta_- f d \lambda = \int_{\partial D_+} \beta_+ f(z) \overline{z} dz + \int_{\partial D_-} -\beta_- f(z) \overline{z} dz =\\
&&\int_{\partial D_+} f(z) S(z) dz + \int_{\partial D_-} f(z )S(z) dz = \langle \mu, f \rangle.
\end{eqnarray*}
(Note that we above used that the boundaries $\partial D_+$ and $\partial D_-$ have opposite orientations on $\partial D_+ \cap \partial D_-$. This is why the quantity $S_+ - S_-$ is relevant on this set.)

Another way to reason is to use the fact that with $\Omega= (\overline{D})^o$, all functions $f$ as above are automatically analytic in $\Omega$ as-well, which for instance follows easily from Morera's theorem. Since the linear span of the Cauchy kernels with poles in $\Omega^c$ are dense in $AL^1(\Omega)$ it follows from this observation that 
$$\beta_+ \int_{D_+} f d \lambda - \beta_- \int_{D_-} f d \lambda = \int f d \mu.$$

Conversely let us assume that $(D_+,D_-)$ is a two-phase quadrature domain with respect to $\mu$, and define 
$$u= U \mu - \beta_+ U (\lambda |_{D_+}) + \beta_- U (\lambda |_{D_-}).$$
By assumption we have $\partial u =0$ in $(\overline{D})^c$, and by continuity this extends to $\overline{((\overline{D})^c)}$. Due to the assumption on $\partial D_{\pm}$ we see that this set contains $\partial D_{\pm} \setminus \Gamma$, and if we define
$$S_{\pm} = \pm \beta_{\pm} \overline{z} - 4 \partial u \textrm{ in } \overline{D}_{\pm},$$
then it is easy to see that these satisfy the requirements in the definition of a two-phase Schwarz function.
\end{proof}
\begin{rem}
Note that it follows from the above proof that a two-phase Schwarz function is uniquely determined close to the boundary, in the sense that two different Schwarz functions for $(D_+,D_-)$ must be equal outside a compact set in $D=D_+ \cup D_-$.
\end{rem}
One contrast with the one-phase case is that as we saw above typically most points of $\partial D_+ \cap \partial D_-$ will be removable singularities for the analytic functions which are continuous up to the boundary. This leads to that there are an abundance of two-phase quadrature domains for analytic functions, even with $\mu=0$ (so called null quadrature domains).
For example let $D_+$ be a simply connected domain with smooth boundary. Let $D_- = \omega(2 \lambda |_{D_+}) \setminus D_+$. Then it is easy to see by definition that we have 
$$\int_{D_+} h d \lambda - \int_{D_-} h d \lambda=0$$
for all functions which are integrable and harmonic in  $\Omega = (\overline{D_+ \cup D_-})^o=\omega(2 \lambda |_{D_+})$. Note that $\partial D_+ \subset \Omega$ under the above circumstances. Furthermore as above all functions $f$ which are continuous in $\Omega$ and analytic in $D_+ \cup D_-$ are automatically analytic in $\Omega$.
Therefore we have 
$$\int_{D_+} f d \lambda - \int_{D_-} f d \lambda =0$$
for all functions which are continuous on $\overline{\Omega}$ and analytic in $D_+ \cup D_-$.
Indeed the assumption that $\partial D_+$ is smooth can be substantially relaxed and hence we can not get any general results regarding the regularity of the boundary just from such a quadrature identity. (The same remark could be said about relaxing the assumption that $u=0$ on $\partial D_+ \cap \partial D_-$ in the definition of a two-phase quadrature domain for harmonic functions.)
This should be compared with the results of \cite{Shahgholian-Uraltseva 2003,Shahgholian-Uraltseva-Weiss 2007} which shows that a two-phase quadrature domain for harmonic functions the boundary locally consists of one or two $C^1$ graphs (but in general not $C^{1,\alpha}$). Also note that on the set $\partial D_{\pm} \setminus \Gamma$ we are in the one-phase situation locally, which is treated in \cite{Sakai 91}, and it means that this part of the boundary is essentially real analytic locally (apart from possibly a finite number of singularities).

About existence of two-phase quadrature domains for harmonic functions, which hence gives plenty of examples of $(D_+,D_-)$ which has a two-phase Schwarz function, there are some results in \cite{Em-Pr-Sh 2010,GS3}. For instance the results in \cite{GS3} implies that if $\mu_+,\mu_-$ both are finite linear combinations of point-masses, then there is a two-phase quadrature domain $(D_+,D_-)$ with respect to $(\mu_+,\mu_-)$.

The simplest ``construction'' of such two-phase quadrature domains (and hence two-phase Schwarz functions) known relies on reflection. We now give an example of this, and then we also give an abstract generalization in terms of the Schottky double due to Bj\"o{}rn Gustafsson.

\begin{example}
Let $z^0 \in H_+$ where $H_{\pm} = \{x \pm iy: y>0\}$. Then we define
$$u = \left\{ \begin{array}{ll} \alpha G_{H_+}(z^0,\cdot) - G_{H_+} (\lambda |_{\omega(H_+,\delta_{z^0})}) & \textrm{ in } H_+\\
0 & \textrm{ in } \R\\
-\alpha G_{H_+}(\overline{z^0},\cdot) + G_{H_-} (\lambda |_{\omega(H_-,\delta_{\overline{z^0}})}) & \textrm{ in } H_-.\end{array}\right.$$
It is easy to verify directly that this makes $(D_+,D_-)= (\omega(H_+,\delta_{z^0}),\omega(H_-,\delta_{\overline{z^0}}))$ a two-phase quadrature domain with respect to $(\mu_+,\mu_-)= (\delta_{z^0},\delta_{\overline{z^0}})$.
(Note that $u$ in the lower half plane is an odd reflection of $u$ from the upper half plane.)
So if we define 
$$S = \left\{ \begin{array}{lr} \beta_+ \overline{z} - 4 \partial u & \textrm{ in } D_+\\
								 -\beta_- \overline{z} - 4 \partial u & \textrm{ in } D_-\end{array}\right.,$$
then we have an example of a two-phase Schwarz function. It should be remarked however that this is the most simple non-trivial example of a two-phase quadrature domain (where trivial means that we have two disjoint one-phase quadrature domains) and even here it seems hard to actually calculate explicitly what the solution is. 
\end{example}

\begin{example}\label{Schottky}
An abstract generalization of the above example can be given using the so called Schottky double. Let $\Omega \subset \C$ be a domain (with sufficiently smooth boundary), $\Gamma \subset \partial \Omega$ a subarc and a positive measure $\mu$ with compact support in $\Omega$. Suppose 
$$u= G_{\Omega} \mu - G_{\Omega} (\lambda |_{\Omega}) \textrm{ in } \overline{\Omega}$$
satisfies $|\nabla u|=0$ on $\partial \Omega \setminus \Gamma$. This is equivalent to saying that for every function $h$ which is continuous on $\overline{\Omega}$, harmonic in $\Omega$ and zero on $\Gamma$ satisfies:
$$\int_{\Omega}h d\lambda = \int h d\mu.$$
Now let $\widetilde{\Omega}$ be a copy of $\Omega$ with the opposite conformal structure, and let $\widehat{\Omega} = \Omega \cup \Gamma \cup \widetilde{\Omega}$ be a partial Schottky double of $\Omega$ welded only along $\Gamma$. We let $\widetilde{\mu}$ be the corresponding ``reflection'' of $\mu$ onto $\widetilde{\Omega}$.
We have a natural bijection $z \mapsto \widetilde{z}$ from $\Omega$ onto $\widetilde{\Omega}$ and vice versa (corresponding to the complex conjugate in the previous example), and if we extend $u$ to $\widehat{\Omega}$ by $u(\widetilde{z}) = -u(z)$ for $z \in \Omega$, then we clearly are in a similar situation as above where we can view $(\Omega,\widetilde{\Omega})$ as a two-phase quadrature domain for harmonic functions with respect to $(\mu,\widetilde{\mu})$. I.e., if $h$ is a function which is continuous on the closure of $\widehat{\Omega}$ (in the obvious sense) and harmonic in $\Omega \cup \widetilde{\Omega}$, then we have
$$\int_{\Omega} h d \lambda - \int_{\widetilde{\Omega}} h d \lambda = \int h d \mu - \int h d \widetilde{\mu}.$$
To see this we simply define
$$h_e(z) = \frac{1}{2} (h(z)+h(\widetilde{z})),$$
$$h_o(z) = \frac{1}{2} (h(z)-h(\widetilde{z})),$$
and then it is easy to see that the statement holds for $h_e$ and $h_o$ separately (in the first case both sides are trivially zero by symmetry and in the second we simply use that $h_o$ is zero on $\Gamma$).
By linearity the statement follows since $h=h_e+h_o$.
\end{example}

\section {Further perspectives}
When it comes to harmonic balls one interesting problem is to develop the corresponding theory of harmonic spheres. This is naturally more difficult, and it should be compared to the corresponding theory of quadrature domains where one also allows surface measures on the boundary of the domain. There are methods to show existence also in this case in e.g. \cite{Gustafsson-Shahgholian 96}. 

For more information on the one-phase Schwarz function we refer to \cite{Davis 74,Shapiro 92}. This function is connected to a large number of other branches of mathematics (the references below are not in any way trying to be complete):

{\it
1)  Moment problems,  Operator Theory (hyponormal operators), Exponential Transforms (see \cite{Gustafsson-Putinar 98,Gustafsson-Putinar 00,McCarthy-Yang 97, Putinar 96, Putinar 98a,Shapiro-Putinar 00,Sakai 78,Xia 96, Xia 03, Yakubovich 01},

2) Hele-Shaw related problems (see \cite{Crowdy 04,Gustafsson-Vasiliev 04,Hedenmalm-Shimorin 02}),

3) Problems related to Bergman and Szeg\"o{} Kernels (see \cite{Bell 03,Bell 04}),

4) Mother bodies and skeletons (see \cite{Gustafsson 98,Gustafsson-Sakai 99,Savina-Sternin-Shatalov 95,Sjodin 02},

5) The Cauchy Problem in $\C^n$ (see \cite{Ebenfelt 93,Johnsson 94,Shapiro 89}),
 
6) Quadrature Surfaces (see \cite{Shahgholian 94,Shahgholian 94b} }.

 It is tantalizing and wishful to think that many of these concepts can be developed also in the two-phase situation. That, however, 
the future will decide.
Here we would also like to point to the possibility to treat two-phase quadrature domains with surface measures as in  \cite{Gustafsson-Shahgholian 96}. Another possibility is to try to look at the problem that corresponds to two-phase quadrature domains for analytic functions in higher dimensions. This would mean that we just assume that the gradient of the function $u$ as defined in (\ref{little u}) in the introduction only satisfies $\nabla u=0$ in $(D \cup \Gamma)^c$ (where again $\Gamma = \partial D_+ \cap \partial D_-$).  Note that just as in two dimensions the boundary need not be very regular for such domains unfortunately.

\end{document}